\documentclass[12pt]{amsart}       
\usepackage{xcolor}

\usepackage[normalem]{ulem}
\usepackage{amsmath}
\usepackage{amsxtra}
\usepackage{amstext}
\usepackage{amssymb}
\usepackage{amsthm}
\usepackage{latexsym}
\usepackage{dsfont} 
\usepackage{verbatim}
\usepackage{rotating}
\usepackage{caption}
\usepackage{subcaption}

\usepackage[shortlabels]{enumitem}
\newtheorem{thm}{Theorem}[section]
\newtheorem{prop}[thm]{Proposition}
\newtheorem{lemma}[thm]{Lemma}
\newtheorem{cor}[thm]{Corollary}
\newtheorem{claim}[thm]{Claim}
\theoremstyle{definition}
\newtheorem{definition}[thm]{Definition}

\newtheorem{remark}[thm]{Remark}
\numberwithin{equation}{section}
\def\R{\mathbb{R}}
\def\S{\mathcal{S}}
\def\eps{\varepsilon}
\def\Lip{\operatorname{Lip}}
\def\Sym{\mathcal{S}_{s,\Omega}}
\def\Symx{\mathcal{S}^x_{s,\Omega}}
\def\F{\mathcal{F}}
\def\wt{\widetilde}
\def\kap{\varkappa}
\def\SLA{\mathcal{SLA}}
\DeclareMathOperator{\dist}{dist} 
\DeclareMathOperator{\supp}{supp}

\DeclareMathOperator{\lip}{Lip}

\def\XXint#1#2#3{{\setbox0=\hbox{$#1{#2#3}{%
\int}$ }
\vcenter{\hbox{$#2#3$ }}\kern-.6\wd0}}

\def\alphaflat{\alpha_{\mu,s}^{\operatorname{flat}}}
\title[Small Local Action]{Small local action of singular integrals on spaces of non-homogeneous type}
\date{\today}

\author[B. Jaye]{Benjamin Jaye}
\address{School of Mathematical and Statistical Sciences, Clemson University, Clemson, SC 29631, USA}
\email{bjaye@clemson.edu}
\author[T. Merch\'an]{Tom\'as Merch\'an}
\address{Department of Mathematical Sciences, Kent State University, Kent, Ohio 44240, USA}
\email{tmercha2@kent.edu}
\thanks{Research supported in part by NSF DMS-1830128 and DMS-1800015.}
\begin{document}

\begin{abstract}Fix $d\geq 2$ and $s\in (0,d)$.  In this paper we introduce a notion called \emph{small local action} associated to a singular integral operator, which is a necessary condition for the existence of principal value integral to exist.  Our goal is to understand the geometric properties of a measure for which an associated singular integral has small local action.  We revisit Mattila's theory of symmetric measures and relate, under the condition that the measure has finite upper density, the existence of small local action to the cost of transporting the measure to a collection of symmetric measures.  As applications, we obtain a soft proof of a theorem of Tolsa and Ruiz-de-Villa on the non-existence of a measure with positive and finite upper density for which the principal value integral associated with the $s$-Riesz transform exists if $s\not\in \mathbb{Z}$.  Furthermore, we provide a considerable generalization of this theorem if $s\in (d-1,d)$.\end{abstract}

\maketitle

\section{Introduction}

The purpose of the present paper and its sequel \cite{JM} is to conduct a study into the relationship between the different ways in which a singular integral operator with nice kernel can act in a space with rough geometry in $\R^d$, $d\geq 2$.

Fix $s\in (0,d)$.  For a Lipschitz continuous, one homogeneous, odd kernel $\Omega$, we form the $s$-dimensional Calder\'{o}n-Zygmund kernel $K(x) = \frac{\Omega(x)}{|x|^{s+1}}$.

Our goal here is  to understand the geometric consequences on a (locally finite, non-negative Borel) measure $\mu$ of a local condition called \emph{small local action}.
 \begin{definition} Fix a family of Lipschitz continuous functions $\Gamma=\{\eta^{\tau}, \tau\in (0,1)\}$ satisfying
$$
\eta^\tau(t) = \left\lbrace
\begin{array}{ll}
1 & \textup{if } 0 \leq t \leq 1-\tau, \\
0 & \textup{if } 1 \leq t < \infty,
\end{array}
\right.
$$
and $\|\eta^\tau\|_{Lip} \leq \frac{1}{\tau}$.
The kernel $\Omega$ has \emph{small local $s$-action}\footnote{We will usually just write small local action, as $s\in (0,d)$ is fixed.}  if
\begin{equation}\tag{$\SLA$}\begin{split}\text{for } \tau \in(0,1),\,  \lim_{r \rightarrow 0}&\; \frac{1}{r^{s+1}} \int_{\R^d} \!\! \Omega(x-y) \eta^\tau\Bigl(\frac{|x-y|}{r}\Bigl)\,d\mu(y) = 0\\& \text{ for }\mu\text{-almost every }x \in \mathbb{R}^d. \end{split}\end{equation}
\end{definition}
Our goal is to understand what ($\SLA$) tells us about $\mu$, under the assumption that $\mu$ has finite upper $s$-density, i.e.
$$\overline{D}_{\mu,s}(x) \stackrel{\operatorname{def}}{=} \limsup_{r\to 0}\frac{\mu(B(x,r))}{r^s}<\infty \; \mu\text{-almost every }x\in \R^d.
$$  Under the finite upper density condition on $\mu$, the property $(\SLA)$ is a necessary condition for the $\mu$-almost everywhere existence of the principal value integral $$\lim_{\eps \to 0}\int_{|x-y|>\eps} \frac{\Omega(x-y)}{|x-y|^{s+1}} d\mu(y),\; x\in \R^d,$$ see Appendix \ref{PVappend}.

Our motivation for introducing the small local action condition was primarily to understand better the difference between the existence of the principal value integral, and the action of the associated Calder\'{o}n-Zygmund operator in $L^2$, and it plays a key role in our paper \cite{JM} on this topic\footnote{We refer the reader to the introduction of \cite{JM} or Tolsa's monograph \cite{To5} for a history of the topic, as it is not our central subject here, but we mention that the existence of principal values is not necessarily implied by the $L^2$ boundedness of the associated operator in a space of non-homogeneous type \cite{CH, Dav} even within the class of homogeneous convolution kernels \cite{JN3}, for additional important results see also \cite{M, NToV, MP,MV, RVT, To3, To4, To5}.}.

In \cite{M},  Mattila already  studied similar properties to $(\SLA)$  under the additional regularity assumption that the measure $\mu$ has \emph{positive lower density:} \begin{equation}\label{PLD}\underline{D}_{\mu,s}(x)\stackrel{\operatorname{def}}{=}\liminf_{r \rightarrow 0} \frac{\mu(B(x,r))}{r^s} >0\text{ for }\mu\text{-almost every }x \in \mathbb{R}^d.\end{equation}  In this article we shall adapt Mattila's machinery to study the property $(\SLA)$ without the lower density assumption.

Certainly, a crude sufficient condition for $(\SLA)$ to hold is that $\mu$ has zero density, i.e. $\overline{D}_{\mu,s}(x)=0$ for $\mu$-almost every $x\in \R^d$.  However, since $\Omega$ is odd, this sufficient condition cannot be necessary if $s \in \mathbb{Z}$. Indeed, if $\mu$ is the induced Lebesgue measure of an $s$-plane, then $(\SLA)$ holds, but $\mu$ does not have zero density.

For $s\in \mathbb{Z}$ a sufficient condition for the property $(\SLA)$ is provided by Lipschitz transportation numbers, introduced to the study of singular integrals by Tolsa (see \cite{To1} and \cite{To2}).  We shall use the following variant of the transportation number
 $$\alphaflat(B(x,r))=\inf_{L \in \mathcal{G}(s,d)}\sup_{\substack{f\in \Lip_0(B(x,4r))\\\|f\|_{\Lip}\leq \tfrac{1}{r}}}\Bigl|\frac{1}{r^s}\int_{\R^d} \varphi\Bigl(\frac{|\cdot-x|}{r}\Bigl) f d(\mu-c_{\mu,L}\mathcal{H}^s_{x+L})\Bigl|,
 $$
 where
 \begin{itemize}
  \item $\mathcal{G}(s,d)$ is the collection of $s$-dimensional linear subspaces of $\R^d$,
 \item  $\varphi$ is a smooth function that satisfies $\varphi\equiv 1$ on $(0,3)$,  $0\leq \varphi \leq 1$, and $\supp(\varphi) \subset (0,4)$,
 \item and $$c_{\mu,L}= \int_{\R^d} \varphi\bigl(\tfrac{|\cdot-x|}{r}\bigl) \,d\mu\Bigl[\int_{\R^d} \varphi \bigl(|\tfrac{\cdot-x|}{r}\bigl) \,d\mathcal{H}^s_{x+L}\Bigl]^{-1}.$$
 \end{itemize}

 The role of the $\alpha$-numbers is exhibited in the following theorem:

 \begin{thm}\label{rieszthm}  Fix $\Omega(x)=x$ to be the Riesz kernel.  Suppose that $\mu$ is a measure with $\overline{D}_{\mu}(x)<\infty$ for $\mu$-almost every $x\in \R^d$.  The small local action property $(\SLA)$ holds if and only if
 \begin{enumerate}
 \item $s\not\in \mathbb{Z}$ and $\mu$ has zero density,
 \item $s\in \mathbb{Z}$ and $\mu$ satisfies $\lim_{r\to 0}\alphaflat(B(x,r))=0$  for $\mu$-a.e $x\in \R^d$.
 \end{enumerate}
 \end{thm}

  The part of Theorem \ref{rieszthm} relating to $s\not\in \mathbb{Z}$ is closely related to (and implies) a theorem of Ruiz de Villa and Tolsa \cite{RVT} on the non-existence of a non-zero measure $\mu$ satisfying  $\overline{D}_{\mu,s}(x)\in (0,\infty)$ for $\mu$-almost every $x\in \R^d$ and for which the $s$-Riesz transform exists in principal value. The proof given by Ruiz de Villa and Tolsa in \cite{RVT} is a delicate analysis which emphasizes the use of specific test functions.  As a byproduct of our work, we obtain a new proof that proceeds via a soft compactness argument.

It is a consequence of Preiss's theorem \cite{P} that if $\underline{D}_{\mu,s}(x)>0$ for $\mu$-almost every $x\in \R^d$ (i.e. (\ref{PLD}) holds), then the condition $\lim_{r\to 0}\alphaflat(B(x,r))=0$ $\mu$-a.e. implies that $\mu$ is $s$-rectifiable (the support of $\mu$ can be covered, up to a sets of zero $s$-dimensional Hausdorff measure, by a countable union of $s$-dimensional Lipschitz submanifolds).  Therefore, under this positive lower density condition, one recovers\footnote{With essentially the same proof.} the Mattila-Preiss theorem \cite{MP} that the existence of the principal value integral of the $s$-Riesz transform implies that the underlying measure is $s$-rectifiable.  However, if $\underline{D}_{\mu,s}(x)=0$ $\mu$-almost every $x \in \R^d$, there are examples of purely unrectifiable measures $\mu$ for which $\lim_{r\to 0}\alphaflat(B(x,r))=0$ $\mu$-a.e. (see Section 5 of \cite{P}), so small local action alone does not imply rectifiability for such irregular measures.  In particular, one cannot expect to recover Tolsa's theorem \cite{To4} on the rectifiability of measures supported on sets of locally finite $s$-dimensional Hausdorff measure for which the Riesz transform exists in principal value from consideration of small local action alone.

Theorem \ref{rieszthm} follows from the general statement Theorem \ref{soft} below, which relates the condition $(\SLA)$ to Mattila's notion of a symmetric measure.

\begin{definition} A point $x \in \mathbb{R}^d$ is an $\Omega$-\emph{symmetric point} of a measure $\nu$ if
 $$ \int_{B(x,r)} \Omega(x-y) \,d\nu(y)=0 \text{  for all } r >0.$$
 The set $\Omega$-symmetric points of a measure $\nu$ is denoted by $S(\Omega,\nu)$.  A measure $\nu$ is called $\Omega$-\emph{symmetric} if $\supp(\nu) \subset S(\Omega,\nu)$.\end{definition}
In our study a subset of symmetric measures will naturally arise.  Define $\mathcal{M}_s$ to be the collection of measures $\mu$ satisfying the growth bound $\mu(B(x,r))\leq r^s$ for all $x\in \R^d$ and $r>0$.  Set
$$\Sym = \{ \nu: \nu \text{ is }\Omega\text{-symmetric}, \; \nu\in \mathcal{M}_s\}.$$
Our main general result relates the property $(\SLA)$ to a certain transportation distance from $\mu$ to the set $\Sym$.  For $x\in \R^d$, set $$\Symx= \{ \nu : \nu \in \Sym, x \in S(\Omega,\nu)\},$$ and
$$ \alpha_{\mu,\Omega,s}(B(x,r)) = \inf_{\nu \in \Symx}\sup_{\substack{f \in \lip_0(B(x,4r)) \\ \|f\|_{\lip}\leq \frac{1}{r}}}
 \Bigl|\frac{1}{r^s}\int_{\R^d} \varphi\Bigl(\frac{|\cdot-x|}{r}\Bigl) f \;
\,d (\mu - c_{\mu,\nu} \nu) \Bigl|,$$
with $c_{\mu,\nu}= \int_{\R^d} \varphi\bigl(\frac{|\cdot-x|}{r}\bigl)d\mu \Bigl[\int_{\R^d} \varphi\bigl(\frac{|\cdot-x|}{r}\bigl) d\nu\Bigl]^{-1}$ if $\int_{\R^d}\varphi \bigl(\frac{|\cdot-x|}{r}\bigl) d\nu \neq 0$ and $c_{\mu, \nu} = 0$ otherwise\footnote{ In the event that $\int_{\R^d}\varphi \bigl(\frac{|\cdot-x|}{r}\bigl) d\nu = 0$, the coefficient $c_{\mu,\nu}$ can be picked to be any real number without changing the results or proofs. }.
\begin{thm}\label{soft}
Suppose that $\overline{D}_{\mu,s}(x)<\infty$ $\mu$-almost every $x \in \R^d$.  The property $(\SLA)$ holds if and only if $\lim_{r \rightarrow 0} \alpha_{\mu,\Omega,s}(B(x,r))=0$ for $\mu$-almost every $x \in \mathbb{R}^d$.
\end{thm}

The novelty in this theorem comes from the fact that no lower regularity conditions on $\mu$ are imposed. The above theorem reduces the study of $(\SLA)$ to the question of understanding the structure of the set $\mathcal{S}_{s,\Omega}$, and the associated set of symmetric points $S(\nu, \Omega)$ for $\nu\in \mathcal{S}_{s,\Omega}$.

 For instance, in order to prove Theorem \ref{rieszthm} above for the Riesz kernel $\Omega(x)=x$, we need to show that the set $\Sym$ consists of only the zero measure for $s\notin\mathbb{Z}$, while if $s\in \mathbb{Z}$, and $x\in \R^d$, then $$\mathcal{S}_{s,\Omega}^x =
\bigl\{\nu\in \mathcal{M}_s : \nu = c\mathcal{H}^s_{x+L}\text{ for some }c>0, \,L\in \mathcal{G}(s,d)\bigl\}.
$$
This result is the content of Proposition \ref{Rieszsymprop} below, which relies on the work by Mattila-Preiss \cite{M, MP} (in the form presented in \cite{JNT}).

We will describe the set $\mathcal{S}_{s,\Omega}$ and associated symmetric points in two further cases: \begin{enumerate}
\item the Huovinen kernel, which is given, for a fixed odd $k\in \mathbb{N}$, by
 \begin{equation}\label{Huoker}\Omega: \mathbb{C}\setminus \{0\} \mapsto \mathbb{C}\setminus \{0\},\; \Omega(z)=\frac{z^k}{|z|^{k-1}},\end{equation}
leading to Theorem \ref{Huovinen} below, and,
\item  non-degenerate, real analytic kernels, in the case when $s\in (d-1,d)$, leading to Theorem \ref{gennonint} below.
\end{enumerate}

Huovinen \cite{H} studied the relationship between the existence of principal value intergrals associated to kernels of the form (\ref{Huoker}) and rectifiability, under the assumption of positive lower density (\ref{PLD}).  This work included a deep study of the symmetric measure associated to the Huovinen kernels, which we revisit in Section \ref{specific} to completely describe the set $\Sym$.   We say that $\nu\in \mathcal{M}_1$ is a $k$-spike measure associated to $L\in \mathcal{G}(1,2)$ and $z \in \mathbb{C}$, if, for some $c\geq 0$,
$$ \nu_{m,L,z}=c\sum_{n=0}^{m-1} \mathcal{H}_{e^{\pi in/m}L+z},
$$
where $m$ divides $k$ (henceforth $m\mid k$).  We set $\text{Spike}_k$ to be the collection of all such spike measures in $\mathcal{M}_1$ over $L\in \mathcal{G}(1,2)$, $z\in \mathbb{C}$, and $m\mid k$.

\begin{thm}[Huovinen kernel] \label{Huovinen}
Fix $\Omega(z)= \frac{z^k}{|z|^{k-1}},$ where $z \in \mathbb{C}\setminus\{0\}$ and $k$ is odd.
For a measure $\mu$ with $\overline{D}_{\mu,s}(z)<\infty$ $\mu$-almost every $z \in \mathbb{C}$, the property $(\SLA)$ holds if and only if
 \begin{enumerate}
\item $s = 1$ and $\mu$ satisfies $\lim_{r\to 0}\alpha_{\mu,\Omega,1}(B(z,r))=0$  for $\mu$-almost every $z\in \mathbb{C}$, where $$\mathcal{S}^z_{1,\Omega}= \{\nu\in \operatorname{Spike}_k: z\in \supp(\nu)\}.$$
\item $s \in (0,2)\setminus \{1\}$ and $\mu$ has zero density, with $$\mathcal{S}_{s,\Omega}=\{\text{the zero measure}\}.$$
  \end{enumerate}
 \end{thm}
This result is applied in our paper \cite[Theorem 1.7]{JM} to give a necessary and sufficient geometric condition on a measure $\mu$ for the Huovinen transform to exist in principal value, assuming the operator is bounded in $L^2(\mu)$, thereby answering a problem which arose in \cite{JN3}. It is important in this application that the only symmetric points of a measure $\nu\in \operatorname{Spike}_k$ are points on the support.  The case of Theorem \ref{Huovinen} for $s\in (1,2)$ is actually a special case of our next result.

  \begin{thm}[Codimension smaller than one]\label{gennonint}
 Fix $s \in (d-1,d)$. Suppose that the function $x\mapsto \Omega(x)|x|^k$ is real analytic for some integer $k$, and the principal value distribution $\frac{\Omega(\cdot\,)}{|\,\cdot\,|^{d+1}}$ has non-vanishing Fourier transform on $\mathbb{S}^{d-1}$ (see Section \ref{codimsec}). For a measure $\mu$ with $\overline{D}_{\mu,s}(x)<\infty$ $\mu$-almost every $x \in \R^d$, the property $(\SLA)$ holds if and only if $\mu$ has zero density.
 \end{thm}
Simple examples show that the non-vanishing condition on the Fourier transform cannot be relaxed, see Remark \ref{remark}.

Since small local action is a necessary condition for principal value, we observe that Theorem \ref{gennonint} provides a substantial generalization of the aforementioned result of \cite{RVT} for operators of co-dimension smaller than one.

\begin{cor}  Fix $s\in (d-1,d)$ and suppose $\Omega$ satisfies the assumptions of Theorem \ref{gennonint}.  If $\mu$ is a Borel measure  satisfying that $\overline{D}_{\mu,s}(x)\in (0,\infty)$ $\mu$-almost every $x \in \R^d$, and the principal value integral associated to $K(x) = \frac{\Omega(x)}{|x|^{s+1}}$ exists $\mu$-almost everywhere, then $\mu$ is the zero measure.
\end{cor}

\section{Preliminaries and Notation}
We begin by listing recurring notation throughout the text.

\subsection{Sets and functions}
\begin{itemize}
\item For $x \in \mathbb{R}^d$
	and $r>0$, $B(x,r)$ denotes the open ball centered at $x$ with radius $r$.
\item $\mathbb{S}^{d-1}$ denotes the unit sphere in $\R^d$.
\item For a set $E\subset \R^d$, we denote by $\Lip_0(E)$ the collection of Lipschitz continuous functions supported in a compact subset of $E$.  $C_0(E)$ denotes the set of continuous functions compactly supported in the interior of $E$.
\item For a function $f$ defined on an open set $U\subset \R^d$, define
$$\|f\|_{\Lip(U)} = \sup_{x,y\in U, \, x\neq y}\frac{|f(x)-f(y)|}{|x-y|}.
$$
In the case $U=\R^d$, we write $\|f\|_{\Lip}$ instead of $\|f\|_{\Lip(\R^d)}$.
	\item Let $a\in\mathbb{R}^d$ and $r>0$. We define the affine map $T_{a,r}:\R^d\to \R^d$ by $T_{a,r}(y)= \frac{y-a}{r}$. For a function $f:\R^d\to \R$ set $f_{x,r}(y)=f\bigl(\frac{|x-y|}{r} \bigl).$
      \item We define the class of functions $\mathcal{F}_{x,r}$ as follows:
      $$\F_{x,r}=\{ f : f \in \Lip_0(B(x,4r)), \|f\|_{\Lip} \leq 1/r\}.$$
\item We denote by $\mathcal{S}(\R^d)$ the Schwartz class functions in $\R^d$, and by $\mathcal{S}'(\R^d)$ the set of tempered distributions.
\end{itemize}

\subsection{Constants}

\begin{itemize}
\item Throughout the paper we shall be considering a fixed Lipschitz continuous, one homogeneous odd kernel $\Omega$.
\item  By $C>0$ we denote a constant that may change from line to line. All constants in the paper can depend on $d$, $s$, the Lipschitz norm of $\Omega$, and the $\|\Omega\|_{L^{\infty}(\mathbb{S}^{d-1})}$ without mention.
\item The symbol $A\lesssim B$ will mean that there exists a constant $C>0$ such that $A \leq CB$.
\end{itemize}

\subsection{Measures}
\begin{itemize}
\item By a measures, we shall always mean a non-negative, locally finite, Borel
	measure.
\item The Lebesgue measure in $\R^d$ is denoted by $m_d$.  The volume element $dm_d(x)$ is often denoted by $dx$.
\item We denote by $\mathcal{H}^s$ the $s$-dimensional Hausdorff measure in $\R^d$.
	\item We denote by $\supp(\mu)$ the closed support of the measure $\mu$; that is,
	$$ \supp(\mu) = \mathbb{R}^d \setminus \{ \cup B: B \text{ is an   open ball with }\mu(B)=0 \}.$$
	\item With $\mathcal{M}_s$ we denote the set of measures with $s$-power growth:
	$$\mathcal{M}_s=\{\nu: \nu(B(x,r))\leq r^s, \text{ for every } x \in \mathbb{R}^d, r>0\}.$$
   \item For $\mu$ a Borel measure, $T:\R^d\to \R^d$ a Borel measurable map, we define the push-forward measure $T_{\#} \mu$ as $T_{\#} \mu (A)=\mu(T^{-1}(A))$ for a Borel set $A\subset \R^d$.
       \item We say that a set $\Gamma$ is $n$-rectifiable if there exist Lipschitz maps $f_j: A_j \subset \R^n \to \R^d$, $j=1,2,...$ such that
$$\mathcal{H}^n\Big(\Gamma\setminus \bigcup_{j=1}^\infty f_j(A_j)\Big)=0.$$
\end{itemize}

\subsection{Weak Convergence} We say that a sequence of measures $\mu_j$ \emph{converges weakly} to a measure $\mu$ if $$\lim_{j\to \infty}\int_{\R^d} f\,d\mu_j = \int_{\R^d} f\, d\mu, \text{ for every }f\in C_0(\R^d).$$
We shall employ the following simple compactness result.

\begin{lemma}\label{weakcompact}\cite[Chapter 1]{M1} If $\mu_j$ is a sequence of measures such that, for every $R>0$,
$$\sup_j \mu_j(B(0,R))<\infty,
$$
then there is a subsequence of the measures that converges weakly to a Borel measure $\mu$. \end{lemma}

\subsection{Basic Remarks} The following three remarks will be regularly used throughout the paper.
\begin{remark}\label{linfinity} For $f\in \F_{x,r}$, and $y\in B(x, 4r)$, $y\neq x$, then $|f(y)| = |f(y)-f(x+4r\tfrac{y-x}{|y-x|})|\leq \frac{1}{r}\cdot 4r=4$, so  $\|f\|_{\infty} \leq 4$.
\end{remark}

\begin{remark}\label{intsym} Suppose $\nu$ is a measure, and $\Omega$ a kernel.  Notice that if $x \in S(\Omega,\nu)$, then for a $\varphi: \mathbb{R} \to \mathbb{R}$ is a Lipschitz continuous function with $\int_0^\infty |\varphi'| \,dr < \infty$ and $\int |\Omega(x-y)| |\varphi(|x-y|)| \,d\nu(y) < \infty$, we have
     $$ \int_{\R^d} \Omega(x-y) \varphi(|x-y|) \,d\nu(y) = -\int_0^\infty  \int_{B(x,r)} \Omega(x-y) \varphi'(r) \,d\nu(y) \,dr=0.$$
\end{remark}

\begin{remark}[Scaling of the $\alpha$-numbers]\label{scalingrem}  If $x\in \R^d$, $r>0$, and $\widetilde\mu = r^{-s}(T_{x,r})_{\#}\mu$.  Then
$$\alpha_{\mu, \Omega,s}(B(x,r)) = \alpha_{\widetilde{\mu}, \Omega,s}(B(0,1)).
$$
\end{remark}

\begin{lemma}\label{nocharge}  Fix $\mu\in \mathcal{M}_s$.  If $n\in \mathbb{N}$ satisfies $0\leq n<s$, then $\mu(\Gamma)=0$ for any $n$-rectifiable set $\Gamma$.
\end{lemma}
\begin{proof}
Since $\mu$ satisfies the growth condition we have that $\mu(A) \leq 2^s C \mathcal{H}^s(A)$ for any $A \subset \R^d$ (see \cite{M1}, Theorem 6.9.). Now recalling that $\mathcal{H}^s(f(A))=0$  for any Lipschitz map $f:A \subset \R^n \to \R^d$ (see \cite{M1}, Theorem 7.5) and that $\mathcal{H}^s<<\mathcal{H}^n$, we achieve the result.
\end{proof}

\section{Proof of Theorem \ref{soft}}
In this section we present the proof of Theorem \ref{soft}.

Firstly, we present a key lemma which shows the relation between the transportation coefficients and the weak convergence of measures.

 \begin{lemma}\label{weak}
 Fix $x\in \R^d$.  Let $\{\mu_j\}_{j\in \mathbb{N}}$ be a sequence of measures that converges weakly to an $\Omega$-symmetric measure $\nu \in \mathcal{S}^x_{s,\Omega}$. Then for any $r>0$,
 $$ \lim_{j \rightarrow \infty} \alpha_{\mu_j,\Omega,s}(B(x,r))=0.$$
 \end{lemma}
 \begin{proof}
 Without loss of generality, we may suppose that $x=0$ and $r=1$ (Remark \ref{scalingrem}).   Certainly $\Lambda := \sup_j\Bigl| \int_{\R^d} \varphi \,d\mu_j \Bigl| < \infty.$

 Fix $\varepsilon>0$. Since the space $\F_{0,1}$  is a relatively compact subset of $C_0(B(0, 4))$, we can find a finite $\varepsilon$-net $f_1,\dots,f_n \in \F_{0,1}$ for some $n\in \mathbb{N}$ (i.e. $\min_{\beta\in \{1,\dots, n\}}\|f-f_{\beta}\|_{\infty}<\eps$ for every $f\in \F_{0,1}$).

 Set $c_j:=c_{\mu_j,\nu}$.  Then by definition, if $\int_{\R^d}\varphi\,d\nu=0$ then $c_j=0$ for every $j\in \mathbb{N}$, while if $\int_{\R^d}\varphi\, d\nu\neq 0$, then $c_j = \frac{\int_{\R^d}\varphi d\mu_j}{\int_{\R^d}\varphi d\nu}\to 1$ as $j \to \infty$.   Either way, there exists $j_0$ such that for every $j> j_0$
 $$ \max_{\beta\in \{1,\dots,n\}} \left| \int_{\R^d} f_{\beta} \varphi \,d(\mu_j-c_j \nu) \right| \leq \varepsilon.$$
 Let $f \in \F_{0,1}$. With $\beta\in \{1,\dots,n\}$ such that $\|f-f_{\beta}\|_{\infty} \leq \varepsilon,$ write
 $$ \Bigl| \int_{\R^d} f\varphi \,d(\mu_j-c_j \nu) \Bigl| \leq   \Bigl| \int_{\R^d} (f-f_{\beta})\varphi \,d (\mu_j-c_j \nu)\Bigl| + \Bigl| \int_{\R^d} f_{\beta} \varphi \,d (\mu_j-c_j\nu) \Bigl|,$$
 so, for sufficiently large $j$,
 $$ \Bigl| \int_{\R^d} f \varphi \,d (\mu_j-c_j \nu) \Bigl|\leq  (\Lambda+2)\varepsilon.$$
 Finally,
 \begin{align*}
 \inf_{\nu' \in \mathcal{S}^x_{s,\Omega}}  \sup_{f\in \F_{0,1}}
 \Bigl|\int_{\R^d}  f \varphi  \;
\,d (\mu_i - c_{\mu, \nu'} \nu') \Bigl| & \leq  \sup_{f\in \F_{0,1}}
 \Bigl| \int_{\R^d} f \varphi \,d(\mu_i-c_{\mu,\nu}\nu) \Bigl| \\ & \leq (\Lambda+2)\eps,
 \end{align*}
 and the lemma is proved.
 \end{proof}
  Now we proceed with the proof of the theorem which uses the machinery of tangent measures introduced by Preiss \cite{P}, and then used in relation to singular integrals by Mattila \cite{M} and Vihtil\"a \cite{V}.

\begin{proof}[Proof of Theorem \ref{soft}]
Assume first that the property $(\SLA)$ holds. 
We can express $\mathbb{R}^d$ as $\bigcup_{k=1}^{\infty} F_k$, except for a set of $\mu$-measure zero, where
$$ F_k := \Bigl\{x \in \mathbb{R}^d : \frac{\mu(B(x,r))}{r^s} \leq k\text{ for every }r\leq \frac{1}{k} \Bigl\}.$$
Fix one of these sets $F_k$ with $\mu(F_k)>0$.  Using Egoroff's theorem we decompose $F_k$ (except for a $\mu$-measure zero set) into Borel sets in which for every $n\in \mathbb{N}$ the convergence $ \frac{1}{r^{s+1}} \int_{\R^d} \Omega(x-y) \eta^{1/n}\Bigl(\frac{|x-y|}{r}\Bigl)\,d\mu(y)\to 0$ as $r\to 0$ is uniform. We look  at any such set $E$ and its intersection with $F_k$.

  Pick $a_0$ to be a density point of $E \cap F_k$ (so $\frac{\mu(B(a_0, r)\cap(E\cap F_k))}{\mu(B(a_0,r))}\to 1$ as $r\to 0$). We claim that $\alpha_{\mu, \Omega,s}(B(a_0,r))\to 0$ as $r\to 0$. Suppose not.  Then there exists a sequence $\{r_j\}_{j\geq1}$ of positive  numbers and $\delta>0$ such that $\lim_j r_j =0$ and $\alpha_{\mu,\Omega,s}(B(a_0,r_j)) > \delta$ for every $j \in \mathbb{N}$.

  We claim that for every $j$, \begin{equation}\label{nondeg}\int_{\R^d}\varphi_{a,r_j}d\mu \geq \frac{1}{8}\delta r^s_j.\end{equation} To see this, first recall that if $f \in \F_{a_0,r_j}$, then $\|f\|_{\infty} \leq 4$ (Remark \ref{linfinity}).  Now let $\nu$ be any measure in 
  $\mathcal{S}_{s,\Omega}^{a_0}$, and choose $f \in \F_{a_0,r_j}$ with $\left|\int f \varphi_{a_0,r_j}
\frac{\,d (\mu - c_{\mu,\nu} \nu)}{r_j^s} \right| \geq \delta$. Then
  \begin{align*}
      \delta\leq  \Bigl| \int_{\R^d} f \varphi_{a_0,r_j}  \frac{\,d(\mu - c_{\mu,\nu} \nu)}{r_j^s} \Bigl| & \leq  \int_{\R^d} |f| \varphi_{a_0,r_j}  \frac{\,d\mu}{r_j^s} + \frac{c_{\mu,\nu}}{r_j^s} \int_{\R^d} |f| \varphi_{a_0,r_j}\,d\nu
      \\& \leq 2\|f\|_{\infty}\frac{1}{r_j^s}\int_{\R^d}\varphi_{a_0,r_j}d\mu,
  \end{align*}
  as required.

   Our next step is to form the scaled measures $\mu_j:=\frac{T_{a_0,r_j\#}\mu}{3^sk r_j^s}$. Then (\ref{nondeg}) becomes
   $\int_{\R^d}\varphi d\mu_j \geq \frac{\delta}{3^s8k}.$  Since $a_0\in F_k$, we certainly have $$ \sup_j \frac{\mu_j (B(0,R))}{R^s}= \sup_{j}\frac{\mu(B(a_0,r_jR))}{3^sk R^sr_j^s} <\infty \text{ for any }R>0.$$ Passing to a subsequence if necessary, we may assume that  $\mu_j$ converges weakly to some Borel measure $\nu$.  We shall show that $\nu\in \mathcal{S}^0_{s,\Omega}$.

  First observe that $\nu$ is a non-zero measure, since
  $$ \int_{\R^d}\varphi d\nu \geq \frac{\delta}{3^s8k}.$$

   Since $a_0$ is a density point of $E \cap F_k$, we have that, for any fixed $M$
  \begin{equation}\label{outsidenomeas}\lim_{j\to \infty}\frac{\mu(B(a_0, M r_j)\backslash (E\cap F_k))}{r_j^s}=0.
  \end{equation}
  Consequently, if $x\in \supp(\nu)$ and $\rho>0$, then
  \begin{equation}\begin{split}\label{maincalc}
    0<\nu(B(x,\rho)) &\leq \liminf_{j \rightarrow \infty} \frac{(T_{a_0,r_j\#} \mu)(B(x,\rho))}{3^sk r_j^s}\\&= \liminf_{j \rightarrow \infty} \frac{\mu(B(a_0+r_jx, \rho r_j))}{3^sk r_j^s} \\
  &\stackrel{(\ref{outsidenomeas})}{=}\liminf_{j \rightarrow \infty} \frac{\mu(E\cap F_k \cap B(a_0+r_jx, \rho
  r_j))}{3^sk r_j^s}
  \end{split}\end{equation}
 In particular, $E \cap F_k \cap B(a_{0}+r_jx,\rho r_j)$ is non-empty
  for all sufficiently large $j.$  Thus we may select a sequence $a_j \in E\cap F_k$  such that
  \begin{equation}\label{xjconv} x_j = \frac{a_j-a_0}{r_j} \rightarrow x.\end{equation}
  Consequently, if $\rho'>\rho$, then appealing to (\ref{maincalc}),
  $$\nu(B(x,\rho))\leq \liminf_{j \rightarrow \infty}\frac{(T_{a_0,r_j\#} \mu)(B(x_j,\rho'))}{3^sk r_j^s} = \liminf_{j \rightarrow \infty} \frac{\mu(B(a_j, \rho' r_j))}{3^sk r_j^s}\leq \frac{\rho'^s}{3^s},
  $$
  where it is used that $\rho'r_j<1/k$ if $j$ is sufficiently large.  Therefore $\nu(B(x,\rho))\leq \frac{\rho^s}{3^s}$ for any $x\in \supp(\nu)$ and $\rho>0$.  This readily implies that $\nu(B(z,r))\leq r^s$ for every $z\in \R^d$ and $r>0$ (i.e. $\nu\in \mathcal{M}_s$).  Indeed, if $\nu(B(z,r))=0$ then there is nothing to prove.  Otherwise pick $z'\in \supp(\nu)\cap B(z,r)$, then $\nu(B(z,r))\leq \nu(B(z', 3r))\leq \frac{(3r)^s}{3^s}=r^s$, as required.

  Our next goal is to show that $\nu$ is a symmetric measure. Fix $x\in \supp(\nu)$ and $x_j$ as in (\ref{xjconv}).  Then for $\rho>0$ and $n\in \mathbb{N}$,
  \begin{align*}
     \int_{\R^d} &\Omega(x-y) \eta^{1/n}\left(\frac{|x-y|}{\rho}\right) \frac{\,d \nu (y)}{\rho^s}  = \lim_{j\to\infty}\int_{\R^d} \Omega(x_j-y) \eta^{1/n}\left(\frac{|x_j-y|}{\rho}\right) \frac{\,d \mu_j (y)}{\rho^s} \\
     &= \lim_{j \rightarrow \infty} \frac{1}{3^sk r_j^s} \int_{\R^d} \Omega(x_j-y) \eta^{1/n}\Bigl(\frac{|y-x_j|}{\rho}\Bigl)  \frac{\,d T_{a_0,r_j \#} \mu(y)}{\rho^s}
     \\
 & =\lim_{j \rightarrow \infty} \frac{1}{3^sk} \int_{\R^d} \Omega\left(\frac{y-a_j}{r_j}\right) \eta^{1/n}\left( \frac{|y-a_j|}{\rho r_j}\right) \frac{\,d \mu
  (y)}{\rho^sr_j^s} =0.
 \end{align*}
As $\chi_{B(x,\rho)}(y)$ is the monotone increasing limit of the sequence $\eta^{1/n}\bigl(\frac{|x-y|}{\rho}\bigl)$ as $n\to \infty$, we infer that
 $$ \int_{B(x,\rho)} \Omega(x-y) \frac{ \,d\nu(y)}{\rho^s} = 0, \text{for every }\rho>0,$$
  i.e. $x\in S(\Omega, \nu)$.  Since $x$ was chosen to be any point on $\supp(\nu)$ we have that $\nu \in \mathcal{S}_{s,\Omega}$.  The same calculation, with $a_j=a_0$ and $x_j=x=0$ for every $j$, shows that $0 \in S(\Omega,\nu)$.  We have verified that $\nu \in \mathcal{S}^{0}_{s,\Omega}$.  Consequently, Lemma \ref{weak} yields that $\lim_{j\to\infty}\alpha_{\mu_j,\Omega,s}(B(0,1))=0$.  However, by assumption $\alpha_{\mu_j,\Omega,s}(B(0,1))\geq \frac{\delta}{3^s8k}$ for every $j$.  This contradiction concludes the proof of this direction of the theorem.

   Now we proceed with the `if' statement of the theorem. Fix $\tau\in (0,1)$,  $\eta^\tau \in \Gamma$, and $\varepsilon >0$.  By assumption, there exists a positive number $r_0\in (0,1)$ such that for every $r < r_0$, $\alpha_{\mu,\Omega,s}(B(x,r)) <
\varepsilon \tau$. Hence for every $r < r_0$ we can find a symmetric measure $\nu$ in $\Symx$ 
such that
$$  \sup_{f \in \F_{x,r}}
\left| \int_{\R^d} f(y) \varphi\left( \frac{|x-y|}{r} \right)
\frac{\,d (\mu - c_{\mu,\nu} \nu)(y)}{r^{s}} \right| <\varepsilon \tau.$$
Insofar as $x\in S(\Omega, \nu)$, $\int_{\R^d}\Omega(x-y)\eta^{\tau}\bigl(\frac{|x-y|}{r}\bigl)d\nu(y)=0$, so
\begin{align*}
  &\left|\int_{\mathbb{R}^d}  \Omega(x-y) \eta^\tau\Bigl(\frac{|x-y|}{r}\Bigl) \frac{\,d \mu(y)}{r^{s+1}} \right|\\
&= \left|\int_{\mathbb{R}^d} \Omega(x-y) \eta^\tau\Bigl(\frac{|x-y|}{r}\Bigl) \frac{\,d (\mu- c_{\mu,\nu} \nu)(y)}{r^{s+1}}
\right| \\
 &=   \left|\int_{\mathbb{R}^d} \Omega(x-y)\eta^\tau\Bigl(\frac{|x-y|}{r}\Bigl)\varphi_{x,r}(y)\frac{\,d (\mu- c_{\mu,\nu}\nu)(y)}{r^{s+1}}
\right| \lesssim C(\tau)\varepsilon,
\end{align*}
where it has been used in the final inequality that the function $y\mapsto \frac{1}{r}\Omega(x-y)\eta^{\tau}\bigl(\frac{|x-y|}{r}\bigl)$ is $\frac{C}{\tau \cdot r}$-Lipschitz.  Letting $\eps\to 0$ we conclude that $(\SLA)$ holds.  This finishes the proof of Theorem \ref{soft}.
\end{proof}

\section{Examining the class $\mathcal{S}_{s,\Omega}$ for particular choices of kernel $\Omega$}\label{specific}

We begin with a simple lemma.

\begin{lemma}\label{nosup}
Suppose $\Omega : \mathbb{R}^d\setminus \{0\} \mapsto \mathbb{R}^d \setminus \{0\}$ is a continuous one-homogeneous kernel.  Fix a measure $\nu$, $x \in S(\Omega, \nu) \setminus \supp(\nu)$ and $d= \dist(x, \supp(\nu))$.

 Then the set $\overline{ B(x,d)} \cap \, \supp(\nu)$ contains at least $2$ points.
\end{lemma}
\begin{proof}
Without loss of generality, we take $x=0$. For the sake of deriving a contradiction, suppose that there exists $x_0 \in \mathbb{R}^d$ such that $\overline{B(0,d)} \cap \, \supp(\nu) = \{x_0\}.$

Since $\Omega$ is continuous and does not vanish, there exists a component $j$ and a positive number $\delta$ such that $|\Omega_j(\omega) | \geq \frac{1}{2} |\Omega_j(x_0)|$ for every $\omega \in B(x_0,\delta)$. 
By elementary metric topology, there exists $\varepsilon \in (0, \delta)$ such that 
$$\supp(\nu) \cap B(0,d+\varepsilon) \subset B(x_0,\delta).$$ Then
\begin{equation*} \begin{split}
\left| \int_{B(0,d+\varepsilon)} \Omega_j(y -x) \,d\nu(y)\right| & =  \left|\int_{B(0,d+\varepsilon)\cap \supp(\nu)} \Omega_j(y -x )\,d\nu(y) \right| \\
 & \geq \frac{1}{2} |\Omega_j(x_0)| \nu(B(x_0,\varepsilon))>0.
\end{split} \end{equation*}
But this is impossible, since $0 \in S(\Omega,\nu)$.
\end{proof}

The lemma immediately yields the following useful corollary.

\begin{cor}\label{nondegplane}
Let $\Omega : \mathbb{R}^d\setminus \{0\} \mapsto \mathbb{R}^d \setminus \{0\}$ be a continuous one-homogeneous kernel.  If $\nu =  \mathcal{H}^s_{|L}$ for an $s$-plane $L$, then $S(\Omega,\nu) = L$.
\end{cor}

\subsection{The Riesz kernel} In view of Theorem \ref{soft}, Theorem \ref{rieszthm} is an immediate consequence of the following description of $\S_{s, \Omega}$.
\begin{prop}\label{Rieszsymprop} Suppose that $\Omega(x)=x$ is the Riesz kernel, then
 $$\S_{s,\Omega}=\begin{cases}\{\nu\in \mathcal{M}_s:\nu = c \cdot \mathcal{H}^s_{|L}, L \text{ an affine }s\text{-plane},0\leq c\}, \text{ if }s\in \mathbb{Z},\\\{\text{the zero measure}\} \text{ if }s\notin \mathbb{Z}.\end{cases}$$ Moreover, if $\nu\in \S_{s,\Omega}$, then $S(\Omega, \nu)=\supp(\nu)$.
\end{prop}
\begin{proof}
  Let $\nu \in \S_{s,\Omega}$ be non-zero. From Proposition 4.7 in \cite{JNT} we see that\footnote{The paper \cite{JNT} uses a weaker notion called $\varphi$-symmetry.  That every symmetric measure is $\varphi$-symmetric follows immediately from Remark \ref{intsym}.} if $\supp(\nu)$ is not contained in an $\lfloor s
  \rfloor$-plane, then for any $\varepsilon >0$, we have that
  $$\lim_{R \rightarrow \infty} \frac{\nu(B(x_0,R))}{R^{\lfloor s \rfloor+1-\varepsilon}}=\infty,$$
  for some $x_0 \in \supp(\nu)$.
  But if $\eps<\lfloor s\rfloor +1-s$, then this estimate contradicts the growth assumption
  \begin{equation}\label{growth}\nu(B(x,r))\leq r^s\text{ for every }x\in \R^d\text{ and }r>0\end{equation} when $r$ is large.

  On the other hand, if $\supp(\nu) \subset L$ for some $\lfloor s \rfloor$-plane $L$,
  then Proposition 4.7 of \cite{JNT} states that either $\nu=c \mathcal{H}^{\lfloor s\rfloor}_{|L}$ or $\supp(\nu)$ is $(\lfloor s
  \rfloor-1)$-rectifiable.   But (\ref{growth}) also implies that $\nu(\Gamma)=0$ for every $(\lfloor s
  \rfloor-1)$-rectifiable set $\Gamma$ (Lemma \ref{nocharge}), and also, if $s\notin \mathbb{Z}$, then $\nu( L)=0$ for any $\lfloor s \rfloor$-plane $L$ (Lemma \ref{nocharge} again).  The description of the set $\S_{s, \Omega}$ is complete.

  The second conclusion of the proposition follows immediately from Corollary \ref{nondegplane}.
\end{proof}

\subsection{The Huovinen kernel}
In his thesis, P. Huovinen \cite{H} developed tools to understand the symmetric measures associated to kernels of the form $$\Omega: \mathbb{C}\setminus \{0\} \mapsto \mathbb{C}\setminus \{0\},$$ $$\Omega(z)=\frac{z^k}{|z|^{k-1}},$$ where $k$ is odd.
First of all, we present the classification of such measures, due to Huovinen \cite{H}, see Theorem 3.26 in \cite{H}.
\begin{thm} \label{Huo}
If $\nu$ is a non-zero $\Omega$-symmetric measure, then one of the following is satisfied:
\begin{enumerate}
\item[(A)] $\nu= c \mathcal{H}^1_{|L}$ for some line $L$ and $c>0$.
\item[(B)] There exist a line $L$, $a \in \mathbb{C}\setminus \{0\}$, and $c,d >0$ such that
 $$\nu= c \sum_{j=-\infty}^\infty \mathcal{H}^1_{|(L + (2j+1)a)} + d  \sum_{j=-\infty}^\infty \mathcal{H}^1_{|(L + 2ja)} .$$
 \item[(C)] There exist $M$ depending only on $\Omega$  and $5 \leq m \leq M$, $x \in \mathbb{C}$, $\alpha \in [0,2\pi)$ and positive numbers $c_0,...,c_{m-1}$ such that
$$\nu =  \sum_{j=0}^{m-1}c_j  \mathcal{H}^1 \lfloor \Lambda'_j , $$ where
$$\Lambda'_j=\{y \in \mathbb{C}: y= x+te^{i(2\pi j/m+\alpha)}, t \in \mathbb{R}^+\}.$$
\item[(D)] Thre exist $0 \leq \alpha < 2\pi$ and $b>0$ such that
$$\supp(\nu) = \bigcup_{j=0}^2 \bigcup_{l=-\infty}^\infty \{y \in \mathbb{C}: y=b\cdot l\cdot e^{i(\alpha + 2j\pi/3 + \pi/2)} + te^{i(\alpha+2j\pi/3)}, t \in \mathbb{R}\}.$$
\item[(E)] $\nu$ is a discrete measure.
\item[(F)] $\nu= \mathcal{L}^2_{|P}$ for some polynomial $P: \mathbb{R}^2 \mapsto \mathbb{R}$.

\end{enumerate}
\end{thm}
Here we revisit some of the arguments of \cite{H} to derive the following precise result, from which (recalling Theorem \ref{soft}), Theorem \ref{Huovinen} is an immediate consequence.

\begin{thm}
A measure $\nu \in \mathcal{S}_{s,\Omega}$ if and only if $\nu\in \mathcal{M}_s$ and

$\bullet$   $s=1$, and $\nu$ is of the form
\begin{enumerate}
\item $\nu =  c \mathcal{H}^1 \lfloor {L}$ for some line $L$ and $c\geq 0$.
\item There exists an odd integer $n$ that divides $k$ such that $3 \leq n \leq k$, $x \in \mathbb{C}$, $\alpha \in [0,2\pi)$ and $c\geq 0$ such that
$$\nu = c \sum_{j=0}^{n-1}  \mathcal{H}^1 \lfloor \Lambda_j , \text{ where }\Lambda_j=\{y \in \mathbb{C}: y=x + te^{i(\pi j/n + \alpha)}, t \in \mathbb{R}\}.$$
Moreover, in either case (1) or (2), $S(\Omega,\nu)=\supp(\nu)$.
\end{enumerate}

$\bullet$ $s \in (0,2) \setminus \{1\}$, and $\nu= \text{the zero measure}$.
\end{thm}

\begin{proof}  We consider a non-zero symmetric measure $\nu$ taking each of the forms $(A)$--$(F)$ from Theorem \ref{Huo} in turn.

It is clear that any measure of the form (A) is symmetric, and from Corollary \ref{nondegplane} we infer that $S(\Omega,\nu)=\supp(\nu)$ for any such planar measure $\nu$.

Any non-zero measure $\nu$ of the form (B) cannot belong to $\mathcal{M}_s$ for any $s \in (0,2)$, since $\nu(B(0,R)$ is of the order $R^2$ for large $R$.

Now assume that $\nu$ is of the form (C).  Without loss of generality, we may assume that $x=\alpha=0$ in (C), so there exists $M$ depending only on $\Omega$ and $5 \leq m \leq M$, $x \in \mathbb{C}$, $\alpha \in [0,2\pi)$ and positive numbers $c_0,...,c_{m-1}$ such that
$$\nu =  \sum_{j=0}^{m-1}c_j  \mathcal{H}^1 \lfloor \Lambda'_j ,\text{ where }\Lambda'_j=\{y \in \mathbb{C}: y=te^{2\pi ij/m}, t \in \mathbb{R}^+\}.$$
We shall prove that $m$ is even. Let $z_0 \in \supp(\nu)\cap \Lambda'_j$. After a rotation and dilation, we may assume that $z_0=1$.
 Let $\theta$ be the angle formed by $\Lambda'_j$ and $\Lambda'_{j+1}$ and $c_{j-1}$ and $c_{j+1}$ the weights associated to $\Lambda'_{j-1}$,  and $\Lambda'_{j+1}$, respectively.

 \begin{figure}
\includegraphics[height=10cm]{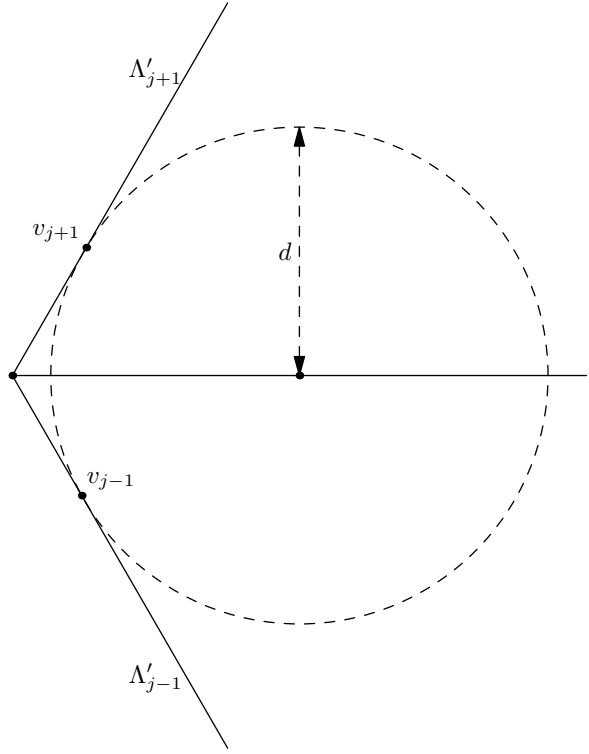}
\caption{The set-up to further analyze the form (C) measures}
\label{fig:typeC1}
\end{figure}

 We denote by $d=\dist(1,\Lambda'_{j-1})=\dist(1,\Lambda'_{j+1})$ and $v_s=\overline{B(1,d)} \cap \Lambda'_s$ for $s=j-1,j+1$.  (See Figure \ref{fig:typeC1}.)  Note that $d<1$ as $m\geq 5$. Note also that $v_{j-1}$ and $v_{j+1}$ are of the form: $$v_{j-1}=\cos \theta e^{i\theta}, v_{j+1}=\cos \theta e^{-i\theta}.$$ Provided that $d+\delta<1$, the oddness of $\Omega$ ensures that the integral
 $$\int_{B(1,d+\delta)\cap \Lambda_j'}\Omega(y-1)d\nu(y)=c_j\int_{(1-(d+\delta), 1+(d+\delta))}\Omega(y-1)d\mathcal{H}^1(y)=0,
 $$
 and so since $\Omega$ is continuous,
 we have that
 \begin{equation} \begin{split} \nonumber
0&= \lim_{\delta \to 0} \frac{1}{\nu(B(1,d+\delta)\cap (\Lambda'_{j+1}\cup \Lambda'_{j-1}))}\int_{B(1,d+\delta))\cap (\Lambda'_{j+1}\cup \Lambda'_{j-1})} \Omega(y-1) \,d\nu(y)\\
& = \lim_{\delta \to 0} \frac{c_{j-1}}{(c_{j-1}+c_{j+1})\mathcal{H}^1(B(1,d+\delta)\cap \Lambda'_{j+1})}\int_{B(1,d+\delta))\cap \Lambda'_{j-1}} \Omega(y-1) \,d\mathcal{H}^1(y) \\
&\phantom{as} + \lim_{\delta \to 0} \frac{c_{j+1}}{(c_{j-1}+c_{j+1})\mathcal{H}^1(B(1,d+\delta)\cap \Lambda'_{j+1}) }\int_{B(1,d+\delta))\cap \Lambda'_{j+1}} \Omega(y-1) \,d\mathcal{H}^1(y)\\
& = \frac{1}{c_{j-1}+c_{j+1}}(c_{j-1} \,\Omega(v_{j-1}-1) + c_{j+1} \,\Omega(v_{j+1}-1)).
\end{split}\end{equation}

Hence we have achieved that
\begin{equation}\label{cjrelation} c_{j-1} (\cos \theta e^{i\theta} -1)^k + c_{j+1}(\cos \theta e^{-i\theta} -1)^k=0.\end{equation}
Solving this, we find that $c_{j-1}=c_{j+1}$ and that
$\frac{\cos \theta e^{i\theta} -1}{|\cos \theta e^{i\theta} -1|}=e^{i\pi l/2k}$, with $l$ odd.
Once more, solving for $\theta$, we obtain that $\theta=\frac{\pi p}{k},$ with $p$ integer.
But $p$ needs to be odd, since otherwise, $e^{i\theta k j}  = 1$ for every $j$ and it would follow that, for any $r>0$,
$$\int_{B(0,r)} \Omega(y) \,d\nu(y)=m c_{j-1} r \neq 0,$$  and $0\in \supp(\nu)$ would not be a symmetric point.
Moreover, we also know that  $m=2\pi/ \left(\frac{\pi p}{k} \right)= \frac{ 2 k}{p} \in \mathbb{Z}$  and, since $p$ and $k$ are both odd, we infer that  $m$ is even. Set $n=m/2$, then $n=k/p$, so $n$ is an odd integer that divides $k$.  We have that $n\geq 3$ since $m=2n\geq 5$.  Now, the fact that $c_{j-1}=c_{j+1}$ for every $j$ ensures that $$\nu = a\sum_{j \text{ even}}  \mathcal{H}^1 \lfloor \Lambda'_j +b\sum_{j \text{ odd}}  \mathcal{H}^1 \lfloor \Lambda'_j ,  $$
for some positive $a$ and $b$.
Since $0\in S(\Omega, \nu)$, we notice that
\begin{equation}\label{weights}
\int_{B(0,r)} \Omega(y) \,d\nu(y)= nr(a-b)=0,
\end{equation}
and so $a=b$.  Therefore $\nu$ is of the form (2).

It remains to prove that $S(\Omega, \nu)= \supp(\nu)$ for a measure $\nu$ of the form (2).  For that purpose we introduce  $$\wt\Lambda=\bigcup_{j=0}^{n-1} \wt\Lambda_j=\bigcup_{j=0}^{n-1} \{ y \in \mathbb{C}: y= te^{i\pi (2j+1)/2n}, t \in \mathbb{R}\},$$ i.e. the union of the bisectors of the support of the measure.

Using Lemma \ref{nosup}, we deduce that the possible symmetric points that lie outside $\supp(\nu)$ must belong to $\wt\Lambda$. Consequently, we readily have that $\supp(\nu) \subset S(\Omega,\nu) \subset \supp(\nu) \cup \wt\Lambda$. Now we prove that indeed, $S(\Omega,\nu)= \supp(\nu)$.
Assume $1 \in S(\Omega,\nu)\cap \wt\Lambda_j$ for some $j$.  Let $\gamma=\frac{  \pi}{2n}$ be the angle formed by $\wt\Lambda_j$ and $\Lambda_j$.

 Mimicking our previous reasoning, we obtain that
$$\frac{1}{2} \Bigl( \Omega(\omega -1) + \Omega(\overline{\omega}-1)\Bigr)=0,$$
where $\omega= \cos \gamma e^{-i\gamma}$.
Consequently, as before, $\gamma=\frac{\pi q}{k},$ with $q$ an odd integer. But then $n=2q/k$ is even, which it isn't. So  $\supp(\nu) = S(\Omega,\nu)$.

On the other hand, arguing as in Example 3.27 of \cite{H}, one readily can see that any measure $\nu$ of the form (2) is symmetric.

Next, suppose $\nu$ is a non-zero symmetric measure of the form (D).  We wish to conclude that necessarily $\nu$ is not in $\mathcal{M}_s$ for $s\in (0,2)$.  Our analysis will repeat the ideas used for type (C) measures.   Notice that the support of $\nu$ is the boundary of a tiling of the plane $\mathbb{C}$ with equilateral triangles, and consists of vertex points with six segments emanating from each vertex point.  See Fig \ref{fig:triangles}

\begin{figure}[h!]
    \includegraphics[width=\linewidth]{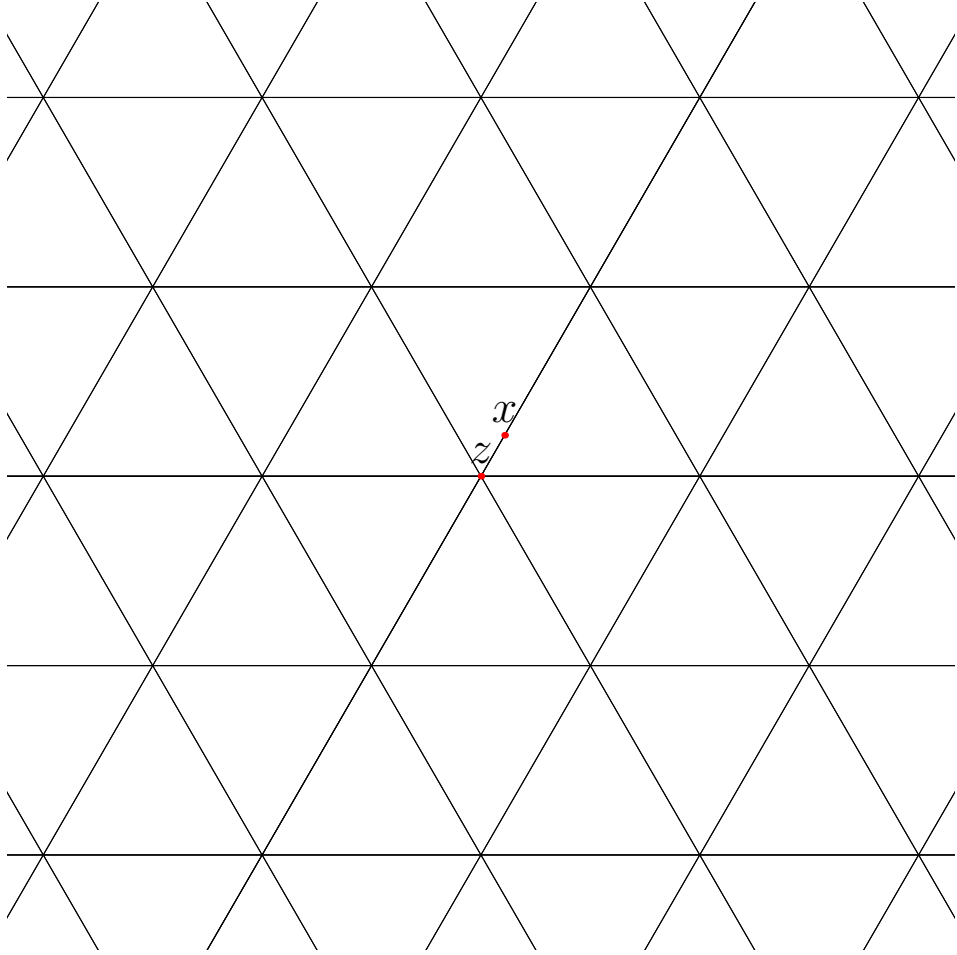}
     \caption{The support of a type (D) measure.}
  \label{fig:triangles}
\end{figure}

From Lemma 3.12 in \cite{H} we infer that $\nu|_{\Lambda} = c_{\Lambda}\mathcal{H}_{\Lambda}$ on each such line segment $\Lambda$, with $c_{\Lambda}>0$.  Now consider a vertex point $z$ and label the six segments through $z$ as $\Lambda_1, \dots, \Lambda_6$.   Fix a segment $\Lambda_i$ and consider the symmetry property at a point $x$ on a segment $\Lambda_j$, $j\in \{1,\dots, 6\}$, with $x$ close to $z$ (see Figure \ref{fig:triangles}).  Repeating the argument leading to (\ref{cjrelation}) (with $x$ replacing $1$, $z$ replacing $0$, and $d=\text{dist}(x, \Lambda_{j-1}\cup\Lambda_{j+1})$, see Figure \ref{fig:typeC1}), we obtain that $c_{\Lambda_{j+1}}=c_{\Lambda_{j-1}}$.  Consequently, there are only two possibilities for the weights $c_{\Lambda_j}$, and for small $r>0$, we have
$$\nu|_{B(z,r)} = \sum_{\substack{j\text{ odd},\\1\leq j\leq 6}}c_1\mathcal{H}^1_{\Lambda_j\cap B(z,r)} +\sum_{\substack{j\text{ even},\\ 1\leq j\leq 6}}c_2\mathcal{H}^1_{\Lambda_j\cap B(z,r)}
$$
for some $c_1,c_2>0$.  However, we may then consider the symmetric property at the vertex point $z$ at some small radius $r>0$.  Then we repeat the calculation in  (\ref{weights}) (with $0$ replaced by $z$) to get that $c_1=c_2$, so $c_{\Lambda_1}=c_{\Lambda_2}=\cdots=c_{\Lambda_6}$. Whence, by a connectivity argument, we find $a>0$ such that $c_{\Lambda}=a$ for any segment $\Lambda$ in the tiling, and $\nu= a \mathcal{H}^1_{| \Theta},$
where$$\Theta= \Bigl\{\bigcup_{j=0}^2 \bigcup_{l=-\infty}^\infty \{y \in \mathbb{C}: y=b\cdot l\cdot e^{i(\alpha + 2j\pi/3 + \pi/2)} + te^{i(\alpha+2j\pi/3)}, t \in \mathbb{R}\} \Bigl\}.$$
Consequently, $\nu(B(0,R))$ is of the order $R^2$ for large $R$ and therefore cannot lie in $\mathcal{M}_s$ for any $s \in (0,2)$.

Certainly discrete measures (type (E)), and measures absolutely continuous with respect to $m_d$ whose density is a non-zero polynomial (type (F)), cannot lie in $\mathcal{M}_s$ for $s\in (0,2)$.

We conclude that if $s=1$ then $\mathcal{S}_{s,\Omega}$ consists of type (A) and type (C) measures, and the set of symmetric points of such a measure is equal to the support, and if $s\in (0,2)\backslash 1$, then $\mathcal{S}_{s,\Omega}$ consists of the zero measure.
\end{proof}

\subsection{Operators of co-dimension less than one}\label{codimsec}

 Given an odd one degree homogeneous kernel $\Omega \in L^1(\mathbb{S}^{n-1})$, we consider the principal value distribution $K(x)=\frac{\Omega(x)}{|x|^{d+1}}$, which acts on a Schwartz class function $\phi \in \mathcal{S}(\mathbb{R}^d)$, by
  \begin{equation*} \begin{split}
  P.V. \int_{\mathbb{R}^d} K(x) \phi(x) \,dx & = \lim_{\varepsilon \to 0}  \int_{B(0,1)\setminus B(0,\varepsilon)} K(x) [\phi(x)-\phi(0)] \,dx\\   & \phantom{asd} +  \int_{B(0,1)^c} K(x) \phi(x) \,dx.
  \end{split}\end{equation*}

  The following classical result may be found in Stein-Weiss \cite{SW}, Theorem 4.7.

  \begin{thm} \label{SW}
 The Fourier transform $m = \widehat{K}$ of the principal value distribution is a function that is homogeneous of degree zero, i.e., $m(x)=m(x/|x|)$ for $x \neq 0$. Moreover
 $$m(x)= - \int_{\mathbb{S}^{n-1}} \Omega(\omega) \Bigl[\frac{i\pi}{2} sgn\Bigl(\frac{x}{|x|} \cdot \omega\Bigl) + \log \Bigl|\frac{x}{|x|} \cdot \omega\Bigl| \Bigl] \,d\mathcal{H}^{d-1}(\omega),$$
 for $x \neq 0$, where $sgn$ denotes the signum function.
  \end{thm}

In our situation, we will be assuming $\Omega$ is a smooth function, in which case $m$ is smooth too, see e.g. Proposition 2.4.8 of \cite{G}.

In view of Theorem \ref{soft}, Theorem \ref{gennonint} follows from the following result.

\begin{thm} \label{sym}
  Suppose that $\Omega$ is an odd, one homogeneous kernel satisfying
  \begin{enumerate}
  \item  there exists $k\in \mathbb{N}$ such that $x\to \Omega(x)|x|^k$ is real analytic in $\R^d$,
  \item  $m(\xi)\neq 0$ for all $\xi \in \mathbb{S}^{d-1}$, where $m$ is the Fourier transform of the principal value distribution associated to $\Omega$.
  \end{enumerate}
  Fix $s \in (d-1,d) $.  If $\mu \in \Sym$, then $\mu \equiv 0$.
\end{thm}

   \begin{proof}[Proof of  Theorem \ref{sym}]     For $t>0$ consider the function
    $$f_{\mu,t}(x) = \int_{\R^d}e^{-t^2|x-y|^2}|x-y|^k\Omega(x-y)\, d\mu(y), \;x\in \R^d.$$
    Since $\mu$ is $\Omega$-symmetric, we may use Remark \ref{intsym} to find that $\supp(\mu)\subset \{f_{\mu,t}=0\}$ for any $t>0$.  Consider the following alternative:
     \begin{enumerate}
       \item $f_{\mu,t} \equiv 0$ in $\mathbb{R}^d$ for every $t >0$, or
       \item there exists some $t_0>0$ such that $f_{\mu,t_0} \not\equiv 0$.
       \end{enumerate}
       Suppose first that  $f_{\mu,t_0} \not\equiv 0$ for some $t_0$.  Then since $f_{\mu,t_0}$ is real analytic on $\R^d$, we have that for every $x\in \R^d$, there is some multi-index $\alpha = (\alpha_1,\alpha_2,\dots, \alpha_d)$, with $\alpha_j\in \mathbb{Z}_+$ such that $D^{\alpha}f_{\mu, t_0}(x)\neq 0$.  For multi-indices $\alpha, \beta$ we write $\alpha<\beta$ to mean $\alpha_j\leq \beta_j$ for every $j\in \{1,\dots, d\}$ but there is some $j\in \{1,\dots, d\}$ with $\alpha_j<\beta_j$.  We write
\begin{align*}
&\supp(\mu) \subset f^{-1}_{\mu,t_0}(0) \cap \bigcup_{\alpha \text{ multi-index}}\{x \in \mathbb{R}^d: D^{\alpha}f_{\mu,t_0}(x) \neq 0 \} \\
& = \bigcup_{\alpha \text{ multi-index}}\{x \in \mathbb{R}^d: D^{\alpha}f_{\mu,t_0}(x) \neq 0 , D^{\beta}f_{\mu,t_0}(x) = 0 \hspace{0.2cm} \text{ for every } \beta <
\alpha\}.
\end{align*}
The implicit function theorem ensures that each set in the union on the right
hand side is locally contained in a smooth $(d-1)$-surface. This contradicts
 the growth of $\mu\in \mathcal{M}_s$ (cf. Lemma \ref{nocharge}).  We can therefore assume that $f_{\mu, t}\equiv 0$ for every $t>0$.\\

 We shall first regularize the measure $\mu$.  Set $g(x) = \int_{\R^d}\varphi(x-y)d\mu(y)$, where $\varphi$ is a smooth non-negative compactly supported function (such as the function $\varphi$ introduced in the definition).  Now define, for a Borel set $A$,
       $$ \mu_{c} (A)= \int_A \int_{\mathbb{R}^d} \varphi(|x-y|) \,d\mu(y) \,dm_d( x) = \int_Ag(x)dm_d(x).$$
       There exists $C>0$ (depending on $\varphi$) such that for every $x\in \R^d$, \begin{equation}\label{smoothgrowthbds}\mu_c(B(x,r)) = \int_{B(x,r)}gdm_d \leq C\min(r^s,r^d) \text{ for any }r>0.\end{equation} Consequently, for $\alpha\in (s+1, d+1)$,  
        \begin{equation}\label{l1mc}\sup_{x\in \R^d}\int_{\R^d}\frac{1}{|x-y|^{\alpha}}d\mu_c(y) = \sup_{x\in \R^d}\int_{\R^d}\frac{1}{|x-y|^{\alpha}}g(y)dy\leq C_{\alpha}. \end{equation}
       
        Besides, $\mu_{c}$ satisfies that $f_{\mu_{c},t} \equiv 0$ on $\mathbb{R}^d$ for every $t>0$.
       Hence, if $\alpha \in (s+1,d+1)$, then
              \begin{align*}
       \int_{0}^{\infty} &t^{\alpha} \int_{\mathbb{R}^d} |x-y|^{k}\Omega(x-y)e^{-t^2|x-y|^2} \,d\mu_c(y) \frac{\,dt}{t}
       \\
       &= \int_{\mathbb{R}^d}  \frac{\Omega(x-y)}{|x-y|^{\alpha}} \int_0^{\infty}
       t^{\alpha + k} e^{-t^2} \frac{\,dt}{t} \,d\mu_c \equiv 0 \text{   in }\mathbb{R}^d.
       \end{align*}
       Since $\| \nabla g \|_{L^{\infty}}\lesssim 1$, we have
              $$  \int_{|x-y|\leq 1} \left| \frac{\Omega(x-y)}{|x-y|^{d+1}}[g(y)-g(x)]\right|
       \,dy \lesssim \int_{|y|\leq 1} \frac{1}{|y|^{d-1}} \,dy \lesssim 1.$$
       Consequently,
       \begin{equation}\begin{split} \label{star}
    & \int_{|x-y| \leq 1} \frac{\Omega(x-y)}{|x-y|^{d+1}} [g(y)-g(x)] \,dy + \int_{|x-y| > 1} \frac{\Omega(x-y)}{|x-y|^{d+1}} g(y) \,dy
     \\
&= \lim_{\alpha \rightarrow d^{-}} \left(\int_{|x-y| \leq 1} \frac{\Omega(x-y)}{|x-y|^{\alpha+1}} [g(y)-g(x)] \,dy
+ \int_{|x-y| > 1} \frac{\Omega(x-y)}{|x-y|^{\alpha+1}} g(y) \,dy \right) \\
&= \lim_{\alpha \rightarrow {d^{-}}} \int_{\R^d} \frac{\Omega(x-y)}{|x-y|^{\alpha+1}} g(y)
\,dy \equiv 0 \text{ in } \mathbb{R}^d,
  \end{split}\end{equation}
  where it was used that $\Omega$ is odd in the third equality.

        Choose $\eta \in S(\mathbb{R}^d)$
      satisfying $\widehat{\eta} \equiv 1$ on $B(0,1)$, $\widehat{\eta} \geq 0$ in $\mathbb{R}^d$
      and $\widehat{\eta} \equiv 0$ outside $B(0,2)$. For $\kap>0$, define $\eta_{\kap}$ by $\widehat{\eta}_{\kap}=\widehat{\eta}
      \left(\frac{\cdot}{\kap}\right)$. Fix $\xi_0 \neq 0$. Since $m$ is smooth (see for example  Propostion 2.4.8 in \cite{G}) and does not vanish on $\mathbb{S}^{d-1}$, we find a component $m_j$ of $m$ and $\kap>0$  for which  $m_j(\xi)\neq0$    for every $\xi \in B(\xi_0,2\kap)$ and $0\notin B(\xi_0, 2\kap)$.

      Fix $\psi = \mathcal{F}^{-1}\widehat{\eta}_{\kap}(\cdot-\xi_0)$ (so that $0\notin \supp(\widehat{\psi})$).  For $\eps\in (0,1)$, set
      $$G_\varepsilon(x)= \int_{|x-y| > \varepsilon} \Delta \psi(y) \frac{\Omega_j(x-y)}{|x-y|^{d+1}}  \,dy.$$
      Then we may write $$G_{\eps}(x)= \int_{\eps<|x-y|\leq 1}[\Delta \psi(y)-\Delta\psi(x)]\frac{\Omega_j(x-y)}{|x-y|^{d+1}}\, dy+\int_{|x-y|>1}\Delta\psi(y)\frac{\Omega(x-y)}{|x-y|^{d+1}}\, dy.$$
      The second integral here is bounded in absolute value by a constant multiple of $\|\Delta\psi\|_{L^1(\R^d)}$.  On the other hand, the function $y\mapsto [\Delta\psi(y)-\Delta\psi(x)]\frac{\Omega_j(x-y)}{|x-y|^{d+1}}$ is bounded by a constant multiple of $\frac{1}{|x-y|^{d-1}}$ (which is locally integrable with respect to Lebesgue measure), so the first integral appearing in $G_{\eps}(x)$ is also bounded, and moreover, as $\eps\to 0^+$, $G_{\eps}$ converges uniformly to the function
      $$G(x)= \int_{|x-y|\leq 1}[\Delta \psi(y)-\Delta\psi(x)]\frac{\Omega_j(x-y)}{|x-y|^{d+1}}\, dy+\int_{|x-y|>1}\Delta\psi(y)\frac{\Omega_j(x-y)}{|x-y|^{d+1}}\, dy.$$
           
           \begin{claim}\label{gepsclaim}We claim that \begin{equation} \label{decay}|G_\varepsilon(x)| \lesssim \frac{1}{1+|x|^{d+2}} \text{   for every }x \in \mathbb{R}^d\text{ and }\eps\in (0, 1/2) .\end{equation}\end{claim}
           
       Assuming the claim we complete the proof of the theorem.   \\
       
      \noindent \textbf{1.} We first prove that $G \ast g \equiv 0$ in $\R^d$. 
       
       For this we want to use (\ref{star}). From the decay estimate (\ref{decay}) and the fact that $g$ is bounded, we appeal to the dominated convergence theorem to yield that
   $$\int_{\R^d} G(x-y) g(y) \,dy  = \lim_{\varepsilon \to 0}  \int_{\R^d} G_\varepsilon(x-y)  g(y) \,dy.$$
 But now, using (\ref{l1mc}) to justify applying Fubini's theorem,
  \begin{align*}
  &\lim_{\varepsilon \to 0}  \int_{\R^d} G_\varepsilon(x-y)  g(y) \,dy \\
   &  \stackrel{\text{Fubini}}{=} \lim_{\varepsilon \rightarrow 0} \iint\limits_{\{(y,z)\in \R^d\times\R^d:\varepsilon < |x-y-z| \leq 1\}}
  \Delta\psi(z)
  \frac{\Omega_j(x-y-z)}{|x-y-z|^{d+1}} g(y) \,dy dz \\
  & \; + \iint\limits_{\{(y,z)\in \R^d\times\R^d:|x-y-z|>1\}} \Delta\psi(z)
  \frac{\Omega_j(x-y-z)}{|x-y-z|^{d+1}} g(y) \,dy \,dz  \\
  & \stackrel{\Omega\text{ is odd}}{=}  \int_{\mathbb{R}^d} \Delta\psi(z) \Bigl[\int_{ |x-y-z| \leq 1}
  \frac{\Omega_j(x-y-z)}{|x-y-z|^{d+1}} [g(y) - g(x-z)] \,dy\Bigl] \,dz \\
 &  \qquad + \int_{\mathbb{R}^d} \Delta\psi(z)\Bigl[\int_{ |x-y-z|>1}
  \frac{\Omega_j(x-y-z)}{|x-y-z|^{d+1}} g(y) \,dy \Bigl]\,dz \stackrel{(\ref{star})}{=} 0,
\end{align*}
and so $G\ast g\equiv 0$ in $\R^d$.\\

   \noindent\textbf{2.}   The next step is to apply Theorem \ref{SW} to show that $\widehat{G}= m \widehat{\Delta \psi}$ in $\mathcal{S}'(\R^d)$. This will in particular show that $G\in \mathcal{S}(\R^d)$ ($0\notin \supp(\widehat{\psi})$).  
   
    To verify this formula, fix $f \in \mathcal{S}(\mathbb{R}^d)$.  The decay estimate (\ref{decay}) (recalling that $g\in L^{\infty}(\R^d)$) certainly ensures that we may apply the dominated convergence theorem to yield
       \begin{align*}
       \int_{\mathbb{R}^d} G(x) \widehat{f}(x) \,d x      &= \lim_{\varepsilon \rightarrow 0}
      \int_{\mathbb{R}^d} \widehat{f}(x)\Bigl[ \int_{|x-y|>\varepsilon} \Delta \psi(y) \frac{\Omega_j(x-y)}{|x-y|^{d+1}}       \,dy \Bigl]\,d x.
      \end{align*}
       For any $\varepsilon >0$,
      \begin{align*}
      \int_{\mathbb{R}^d} &\widehat{f}(x) \int_{|x-y|>\varepsilon} \Delta \psi(y) \frac{\Omega_j(x-y)}{|x-y|^{d+1}} \,dy      = \int_{|y|>\varepsilon} \frac{\Omega_j(y)}{|y|^{d+1}} (\Delta \psi(-\,\cdot\,) \ast \widehat{f} )(y)
      \,dy.
      \end{align*}
            Notice that $\Delta \psi (-\,\cdot\,) \ast \widehat{f}= \mathcal{F} (\widehat{\Delta \psi} \cdot f)$. So if we denote $h= \widehat{\Delta \psi} \cdot f$, then by applying Theorem \ref{SW}, we obtain
      $$ \operatorname{P.V.} \int \frac{\Omega_j(y/|y|)}{|y|^d}\widehat{h}(y)
      \,dy = \int_{\R^d} m_j(y) \widehat{\Delta\psi}(y) f(y)
      \,dy.$$

      Consequently $\widehat{G} = m_j \widehat{\Delta\psi}  \in \mathcal{S}'(\mathbb{R}^d),$ as claimed.\\

\noindent\textbf{3.} We next show that the support of $\widehat{g}$ is contained in $\{0\}$.  

To this end, note that we can express $\widehat{G}$ as
$$\widehat{G}(\xi)=b|\xi|^2 \widehat{\eta}_{\kap}(\xi - \xi_0) m_j(\xi), \text{ for some }b\in \mathbb{C}.$$
Hence $\widehat{G}(\xi)\neq 0$ in $B(\xi_0,t)$. Let $\varepsilon \in (0,\kap/2)$ and
consider the function $F\in \mathcal{S}(\R^d)$ given by
$$ \widehat{F} = \frac{\widehat{\eta}_{\varepsilon}(\xi-\xi_0)}{\widehat{G}(\xi)}.$$
Since
$$|G| \ast g(x) = \int_{\R^d} |G(x-y)| g(y) \,dy \lesssim \int_{\R^d} \frac{1}{1 + |x-y|^{d+2}} \,d\mu_c(y) \lesssim 1,$$
we achieve that $[|F| \ast (|G| \ast g)] (x) < \infty$ for every $x$.
Thus $(F \ast G) \ast g = F \ast ( G \ast g) \equiv 0$ in $\mathbb{R}^d.$
But considering that $F \ast G = \mathcal{F}^{-1}(\widehat{\eta}_{\varepsilon}(\cdot -
\xi_0))$,
we obtain $[ \mathcal{F}^{-1}(\widehat{\eta}_{\varepsilon}(\cdot -
\xi_0))] \ast g \equiv 0.$
So we deduce that $\widehat{g}$ vanishes in the ball $B(\xi_0, \varepsilon)$.
Since $\xi_0$ is arbitrary, $\supp(\widehat{g})\subset \{0\}$.\\

\noindent\textbf{4.} We now complete the proof of the theorem by showing that $\mu$ is the zero measure.   Since  $\supp(\widehat{g})\subset \{0\}$, we have that
$$ \mu_c = P m_d,$$
for some polynomial $P$.
If the polynomial is non-zero, there is a constant $c>0$ such that for all
sufficiently large $R$, $\mu_c(B(0,R))\geq cR^d$. But due to the power growth, $\mu_c(B(0,R)) \lesssim R^s$ for large $R>0$.
Hence $P\equiv  0$.  This then implies that $\mu_c$, and therefore $\mu$, is the zero measure.  The theorem is proved.
\end{proof}

We now return to supply the proof of Claim \ref{gepsclaim}.

\begin{proof}[Proof of Claim \ref{gepsclaim}]
 The observations in the paragraph prior to the statement of Claim \ref{gepsclaim} ensur that $G_{\eps}$ is bounded, so we may assume that $|x|>1$.  We write
      \begin{align*}
      G_\varepsilon(x)= & \int_{|y|>\varepsilon} \Delta \psi(x-y) \widehat{\eta}\left(\frac{4(y-x)}{|x|} \right)
      \frac{\Omega_j(y)}{|y|^{d+1}} \,dy \\
      &+ \int_{|y|>\varepsilon} \Delta \psi (x-y) \left[1 - \widehat{\eta}\left(\frac{4(y-x)}{|x|} \right)  \right]\frac{\Omega_j(y)}{|y|^{d+1}}  \,dy \\
      & = I +II.
      \end{align*}
      Notice that
      \begin{equation}\label{etasupp1}
      \widehat{\eta}\left(\frac{4(y-x)}{|x|} \right) K_j(y) \text{ is supported in }B(x,|x|/2).
      \end{equation}
Therefore, the support of the integrand in $I$ does not intersect $\{|y|\leq \eps\}$ and so we may integrate by parts to obtain
      $$|I|=\left| \int_{B(x, \frac{|x|}{2})}  \psi(x-y) \Delta_y \left[ \widehat{\eta}\left(\frac{4(y-x)}{|x|} \right) K_j(y) \right]
      \,dy\right|,$$
      where $K_j(x)=\frac{\Omega_j(x)}{|x|^{d+1}}$.
      But for $y \in B(x,\frac{|x|}{2})$, we
      have that
      $$ \left| \Delta_y \left[ \widehat{\eta}\left(\frac{4(y-x)}{|x|} \right) K_j(y) \right] \right|  \lesssim \frac{1}{|x|^{d+2}},$$
      and so
      $$\left|  \int_{\mathbb{R}^d}  \psi(x-y) \Delta_y \left[ \widehat{\eta}\left(\frac{4(y-x)}{|x|} \right) K_j(y) \right]
      \,dy \right| \lesssim \frac{1}{|x|^{d+2}} \|\psi \|_{L^1} \lesssim \frac{1}{|x|^{d+2}}.$$
      To deal with the term $II$, notice first that $y \mapsto 1 -  \widehat{\eta}\left(\frac{4(y-x)}{|x|} \right)$
      is supported in $\mathbb{R}^d \setminus B(x,\frac{|x|}{4})$, so we may write      \begin{align*}
       &  \left| \int_{\mathbb{R}^d \setminus B(0,\varepsilon)}
 \Delta_y(\psi(x-y)) \left[1 - \widehat{\eta}\left(\frac{4(y-x)}{|x|} \right)  \right] K_j(y) \,dy \right|
 \\
& \qquad  \leq \left|  \int_{B\left(0,\frac{|x|}{4}\right) \setminus B(0,\varepsilon)} \Delta_y \psi(x-y)     \left[1 - \widehat{\eta}\left(\frac{4(y-x)}{|x|} \right)  \right] K_j(y) \,dy \right|\\ &  \qquad \;+
  \left| \int_{\mathbb{R}^d \setminus (B\left(0,\frac{|x|}{4}\right) \cup B(x,\frac{|x|}{4}))}\Delta_y \psi(x-y)      \left[1 - \widehat{\eta}\left(\frac{4(y-x)}{|x|} \right)  \right] K_j(y) \,dy \right|\\&\qquad \; =\widetilde{I}+ \widetilde{II}.
      \end{align*}
      Now, since $\psi \in \mathcal{S}(\mathbb{R}^d)$, we have that
      $$| \Delta \psi (x-y)| +|\nabla \Delta\psi(x-y)|\leq \frac{C_n}{|x-y|^n}, $$  for every $n \in \mathbb{N}.$         Consider now $\widetilde{I}$.  Recalling (\ref{etasupp1}) we write
        \begin{equation*} \begin{split}
        \widetilde{I} &=\left|\int_{B\left(0,\frac{|x|}{4}\right)\setminus B(0,\varepsilon)} \Delta_y\psi(x-y) K_j(y) \,dy\right|\\&=\left| \int_{B\left(0,\frac{|x|}{4}\right)\setminus B(0,\varepsilon)} [\Delta_y\psi(x-y)-\Delta_y\psi(x)] K_j(y) \,dy\right|\\ & \lesssim \int_{B\left(0,\frac{|x|}{4}\right)} \frac{1}{|y|^{d-1}} \sup_{z \in B\left(0,\frac{|x|}{4}\right)} |\nabla \Delta \psi (x-z)| \,dy \lesssim \frac{C_n}{|x|^{n-1}}.
        \end{split}\end{equation*}
  Regarding $\widetilde{II}$, we notice that $|K_j(y)| \lesssim \frac{1}{|x|^d}$ on the domain of integration. Consequently, if $n>d$, then
  $$\widetilde{II} \lesssim \frac{1}{|x|^d} \int_{|x-y| > \frac{|x|}{4}} \frac{1}{|x-y|^n} \,dy \lesssim \frac{C_n}{|x|^n}.$$
  Putting $n=d+3$ yields the claimed estimate (\ref{decay}).\end{proof}

\begin{remark} \label{remark}
We make note that the Fourier condition in the theorem is sharp: let $d=2$ and $s \in (1,2)$. We consider the kernel $\Omega$:
$$\Omega(x)=x_1 \text{ for }x=(x_1,x_2)\in \R^2.$$
It is clear that the Fourier transform of the associated Principal Value distribution  vanishes on $\{\xi=(\xi_1,\xi_2)\in \R^2:\xi_1=0\}$.
We form the measure $\mu=m_1 \times \mathcal{H}_{|C}^{s-1},$
where $C \subset \mathbb{R}$ is a set with $\mathcal{H}^{s-1}(C)>0$ and $\mathcal{H}^{s-1}(C\cap (x-r,x+r))\leq r^{s-1}$ for every $x\in \R$ and $r>0$ (this can be accomplished with a standard Cantor set construction).  Then $\mu$ has power growth (i.e. $\mu(B(x,r)) \leq r^s$ for every $r>0$, $x \in \mathbb{R}^d$) and $\mu$ is $\Omega$-symmetric.
\end{remark}

\appendix
\section{\\From principal value Integral to small Local action}\label{PVappend}
In this section, we prove that the property $(\SLA)$ is a necessary condition for the almost everywhere existence of the principal value integral.

\begin{prop}\label{PVtoSLA}
Let $\mu$ be a measure satisfying $\overline{D}_{\mu,s}(x) < \infty$ for $\mu$-almost every $x \in \R^d$. If the principal value integral $$\lim_{\eps \to 0}\int_{|x-y|>\eps} \frac{\Omega(x-y)}{|x-y|^{s+1}} d\mu(y),$$ exists for $\mu$-almost everywhere $x \in \R^d$ then $\mu$ satisfies the property $(\SLA)$.
\end{prop}
\begin{proof}
We assume that $\overline{D}_{\mu,s}(0)<\infty$, and the principal value integral exists at $0$, and will verify that for every $\psi\in \Lip_0([0,\infty))$ (a Lipschitz continuous function supported in a compact subset of $[0,\infty)$),
$$\lim_{r\to 0}\int_{B(0,r)}\Omega(y)\psi\Bigl(\frac{|y|}{r}\Bigl)d\mu(y)=0,$$
from which the result follows.    There exists  $k\in \mathbb{N}$ such that
 $ \frac{\mu(B(0,r))}{r^s} \leq k\text{ for every }r\leq \frac{1}{k}$.
 Fix $\eps, \delta\in (0,\tfrac{1}{2})$.   Appealing to the existence of the principal value integral, we choose $R_0\in (0,\tfrac{1}{k})$  such that for any $r_1,r_2 \in (0,R_0)$  we have that
 \begin{equation}\label{pv}\Big| \int_{r_1 \leq |y| < r_2} \frac{\Omega(y)}{|y|^{s+1}} \,d\mu(y) \Big| \leq \delta.\end{equation}
 Fix $r\in (0,R_0)$.  Set $r_j=(1-\varepsilon)^jr$, $j\geq 0$.  Our goal is to estimate
 $$\Bigl|\int_{B(0,r)} \frac{ \Omega(y)}{r^{s+1}} \,d\mu(y)\Bigl|,$$
 which, by the triangle inequality, is no greater than
 \begin{equation*} \begin{split}\Big| & \sum_{j=0}^{\infty}  (1-\varepsilon)^{j(s+1)}\int_{r_{j+1} \leq |y| < r_j} \Omega(y) \Big(\frac{1}{r_j^{s+1}}-\frac{1}{|y|^{s+1}} \Big) \,d\mu(y)\Big|\\
& \phantom{asdf} +  \Big| \sum_{j=0}^{\infty} (1-\varepsilon)^{j(s+1)} \int_{r_{j+1} \leq |y| < r_j} \frac{\Omega(y)}{|y|^{s+1}} \,d\mu(y) \Big|= I+II.
 \end{split}\end{equation*}
 Regarding $I$, by estimating the derivative of the function $y \mapsto \frac{1}{|y|^s}$, we observe that
 \begin{equation*} \begin{split}I &\lesssim  \varepsilon  \sum_j (1-\varepsilon)^{j(s+1)} \int_{r_{j+1} \leq |y| < r_j} |\Omega(y)| \frac{r_j}{r_{j+1}^{s+2}}\,d\mu(y) \\
 & \lesssim   \varepsilon \sum_j \frac{\mu(B(0, r_j)\backslash B(0, r_{j+1}))}{r^s}\lesssim \eps \frac{\mu(B(0,r))}{r^s} \lesssim \varepsilon (s+1)k. \end{split}\end{equation*}
 For $II$ we apply (\ref{pv}) to each of the integrals in the sum to infer that $II \lesssim  \frac{\delta}{\varepsilon}.$  Therefore
 $$ \Bigl| \int_{B(0,r)} \frac{\Omega(y)}{r^{s+1}} \,d\mu(y)\Bigl|\lesssim \eps k  +\frac{\delta}{\eps},$$
 whence $\lim_{r\to 0}\frac{1}{r^{s+1}}\int_{B(0,r)}\Omega(y)d\mu(y)=0$.

 Finally, since a function $\psi\in \Lip_0([0,\infty))$ is compactly supported and absolutely continuous, we may write
 $$ \int_{\R^d} \frac{\Omega(y)}{r^{s+1}} \psi\Big(\frac{|y|}{r}\Big) \,d\mu(y)= - \int_{0}^{\infty} \psi'(t) \Bigl[ \int_{B(0,tr)} \frac{\Omega(y)}{r^{s+1}} \,d\mu(y)\Bigl] \,d t.$$
 Since $\psi'$ is bounded with compact support, we infer that
 $$\lim_{r\to 0}\int_{\R^d}\frac{\Omega(y)}{r^{s+1}}\psi\Bigl(\frac{|y|}{r}\Bigl) \, d\mu(y)=0,
 $$
 as required.
 \end{proof}

\end{document}